\documentclass[bibalpha]{amsart}
\usepackage{geometry} 
\usepackage{verbatim}
\usepackage{enumerate}
\usepackage[mathscr]{euscript}

\newcommand{\boundr}[1]{\ensuremath{\mathrm{Bd}}(#1)}

\def \cl {\operatorname{Cl}}

\def \N{\mathbb{N}}

\def \R{\mathbb{R}}

\def \Z{\mathbb{Z}}

\newcommand{\rvec}{\mathbb{R}_{\mathrm{Vec}}}
\newcommand{\dimcm}{\dim_{\mathrm{CM}}}

\def \cl {\mathrm{Cl}}
\usepackage{enumerate}
\newtheorem{Th}{Theorem}[section]
\newtheorem{Thm}[Th]{Theorem}

\newtheorem{Fact}[Th]{Fact}

\newtheorem{Lem}[Th]{Lemma}

\newtheorem*{Lem*}{Lemma}

\begin{document}
\title{Coarse dimension and definable sets in expansions of the ordered real vector space}

\author{Erik Walsberg}
\address{Department of Mathematics, Statistics, and Computer Science\\
Department of Mathematics\\University of California, Irvine, 340 Rowland Hall (Bldg.\# 400),
Irvine, CA 92697-3875}
\email{ewalsber@uci.edu}
\urladdr{http://www.math.illinois.edu/\textasciitilde erikw}

\date{\today}

\maketitle

\begin{abstract}
Let $E \subseteq \R$.
Suppose there is an $s > 0$ such that
$$ | \{ k \in \mathbb{Z}, -m \leq k \leq m - 1 : [k,k+1] \cap E \neq \emptyset \} | \geq m^s $$
for all sufficiently large $m \in \N$.
Then there is an $n \in \N$ and a linear $T : \R^n \to \R$ such that $T(E^n)$ is dense.
It follows that if $E$ is in addition nowhere dense then $(\R,<,+,0,(x \mapsto \lambda x)_{\lambda \in \R}, E)$ defines every bounded Borel subset of every $\R^n$.
\end{abstract}

\section{Introduction}
\noindent
Let $X \subseteq \R^n$ be bounded and $Z \subseteq \R^n$.
Given a positive $\delta \in \R$ we let $\mathscr{M}(\delta,X)$ be the minimum number of open $\delta$-balls required to cover $X$.
Equivalently $\mathscr{M}(\delta,X)$ is the minimal cardinality of a subset $S$ of $X$ such that every $x \in X$ lies within distance $\delta$ of some element of $S$.
Let $B_n(p,r)$ be the open ball in $\R^n$ with center $p$ and radius $r > 0$ and let $B_n(r) = B_n(0,r)$.
We define the \textbf{coarse Minkowski dimension} of $Z$ to be
$$ \dimcm(Z) := \limsup_{r \to \infty} \frac{\mathscr{M}(1, B_n(r) \cap Z )}{\log(r)}. $$
It is easy to see that the coarse Minkowski dimension of $Z$ is bounded above by $n$ and the coarse Minkowski dimension of a bounded set is zero.
An application of the first claim of Fact~\ref{fact:metric0} below shows that replacing one with any fixed real number $\delta > 0$ does not change the coarse Minkowski dimension.
A simple computation shows that $\dimcm(Z)$ is the infimum of the set of positive $s \in \R$ such that $\mathscr{M}(1,B_n(r) \cap Z) < r^s$ for all sufficiently large $r > 0$.

\medskip \noindent 
We define
$$ \mathscr{N}(X) := \left| \left\{ (k_1,\ldots,k_n) \in \Z^n : X \cap \prod_{i = 1}^{n} [k_i, k_i + 1] \neq \emptyset   \right\} \right|.  $$
It is well-known and easy to see that there is a real number $K > 0$ depending only $n$ such that
$$ K^{-1} \mathscr{M}(1,X) \leq \mathscr{N}(X) \leq K \mathscr{M}(1,X). $$
So
$$ \dimcm(Z) = \limsup_{r \to \infty} \frac{\mathscr{N}(B_n(r) \cap Z )}{\log(r)}. $$
Our main geometric result is Theorem~\ref{thm:wedge}.

\begin{Thm}
\label{thm:wedge}
Suppose $E \subseteq \R$.
If $\dimcm(E) > 0$ then $T(E^n)$ is dense for some $n \in \N$ and linear $T : \R^n \to \R$.
Equivalently, if there is a positive $s \in \R$ such that $\mathscr{N}(B_1(r) \cap E) \geq r^s$ for all sufficiently large $r \in \R$ then there exist $n \in \N$ and linear $T : \R^n \to \R$ such that $T(E^n)$ is dense.
\end{Thm}

\noindent
The converse implication to Theorem~\ref{thm:wedge} does not hold.
Let $D = \{ 2^n, 2^n + n : n \in \N \}$.
A simple computation shows that $D$ has coarse Minkowski dimension zero.
Let $S : \R^4 \to \R$ be given by $S(x_1,x_2,x_3,x_4) = (x_1 - x_2) + \alpha (x_3 - x_4)$ for a fixed $\alpha \in \R \setminus \mathbb{Q}$.
Then $S(D^4)$ is dense.

\medskip\noindent 
Theorem~\ref{thm:wedge} is motivated by an application to logic that we now describe.
Let $\rvec$ be the ordered vector space $(\R,<,+,0,( x \mapsto \lambda x)_{\lambda \in \R})$ of real numbers.
For any subset $E$ of $\R$ let $(\rvec,E)$ be the expansion of $\rvec$ by a unary predicate defining $E$.
When we say that a subset of $\R^n$ is definable in a first order expansion of $(\R,<,+,0)$ such as $(\rvec,E)$ we mean that it is first order definable possibly with parameters from $\R$.

\medskip \noindent
Hieronymi and Tychonievich~\cite{HT} show that $(\rvec,\mathbb{Z})$ defines all bounded Borel subsets of all $\R^n$.
In contrast, it follows from \cite{ivp,weis} that every subset of $\R^n$ definable in $(\R,<,+,0,\Z)$ is a finite union of locally closed sets.

\medskip \noindent 
The theorem of Hieronymi and Tychonievich is a special case of Theorem~\ref{thm:FM}.
Theorem~\ref{thm:FM} also follows from a more general theorem of Fornasiero, Hieronymi, and Walsberg \cite[Theorem 7.3, Corollary 7.5]{FHW-Compact}.
We let $\cl(E)$ be the closure of $E\subseteq \R$ and $\boundr E$ be the boundary of $E$. 
Recall that the boundary of a subset of $\R$ is always closed.

\begin{Thm}\label{thm:FM}
Suppose that $E \subseteq \R$ is not dense and co-dense in any nonempty open interval.
Then the following are equivalent:
\begin{enumerate}
    \item $(\rvec,E)$ does not define every bounded Borel subset of every $\R^n$,
    \item Every subset of $\R$ definable in $(\rvec,E)$ either has interior or is nowhere dense,
    \item $T(\boundr E^n)$ is nowhere dense for every linear $T : \R^n \to \R$.
\end{enumerate}
\end{Thm}
\noindent The implication $(3) \Rightarrow (2)$ is a corollary of a result of Friedman and Miller~\cite{FM-Sparse}.
The implication $(1) \Rightarrow (3)$ is a corollary of the main theorem of \cite{HT}.
Note that $\boundr E$ is nowhere dense as $E$ is not dense and co-dense in any open interval.
If $E$ is bounded then $(3)$ above is equivalent to a natural geometric condition on $E$.
This equivalence, observed in \cite[Theorem 7.3]{FHW-Compact}, is an easy consequence of the famous Marstrand projection theorem (see Mattila~\cite[Chapter 9]{Mattila}) and the classical theorem of Steinhaus that $Z - Z := \{ z - z' : z,z' \in Z \}$ has interior whenever $Z \subseteq \R^n$ has positive $n$-dimensional Lebesgue measure.

\begin{Fact}\label{fact:hasudorff}
Suppose $F \subseteq \R$ is bounded.
Then $T(F^n)$ is nowhere dense for every linear $T : \R^n \to \R$ if and only if $\cl(F)^n$ has Hausdorff dimension zero for all $n \in \N$.
\end{Fact}

\noindent Fact~\ref{fact:hasudorff} does not hold for unbounded subsets of $\R$.
The set of integers, like any countable set, has Hausdorff dimension zero, and $T(\Z^2)$ is dense for any linear $T : \R^2 \to \R$ of the form $T(x,y) = x + \alpha y$ with $\alpha \in \R \setminus \mathbb{Q}$.
Combining Theorem~\ref{thm:wedge} and Theorem~\ref{thm:FM} we obtain the following.

\begin{Thm}\label{thm:main}
Suppose $E \subseteq \R$ is not dense and co-dense in any nonempty open interval.
If $\boundr{E}$ has positive coarse Minkowski dimension then $(\rvec,E)$ defines every bounded Borel subset of every $\R^n$.
In particular if $E$ is nowhere dense and has positive coarse Minkowski dimension then $(\rvec,E)$ defines every bounded Borel subset of every $\R^n$.
\end{Thm}

\noindent
Note that $\Z$ has coarse Minkowski dimension one so Theorem~\ref{thm:main} generalizes the result of Hieronymi and Tychonievich described above.
There are subsets $E$ of $\R$ with coarse Minkowski dimension zero such that $(\rvec,E)$ defines every bounded Borel subset of every $\R^n$ such as $\{ 2^n, 2^n + n : n \in \N\}$ (see the comment after Theorem~\ref{thm:wedge}).
Theorem~\ref{thm:main} fails without the assumption that $E$ is not dense and co-dense in any nonempty open interval.
Block-Gorman, Hieronymi, and Kaplan~\cite{GoHi-Pairs} show that every closed subset of $\R^n$ definable in $(\rvec,\mathbb{Q})$ is already definable in $\rvec$ and $\boundr{\mathbb{Q}} = \R$ has coarse Minkowski dimension one.

\medskip \noindent
The present paper is part of the broader study of the metric geometry of definable sets in first order structures expanding $(\R,<,+,0)$, see \cite{FHM, FHW-Compact, HM}.
Fornsiero, Hieronymi, and Miller~\cite{FHM} show that if $E \subseteq \R$ is nowhere dense and has positive Minkowski dimension then $(\R,<,+,\cdot,0,1,E)$ defines every Borel subset of every $\R^n$.
This statement fails over $\rvec$, as $D = \{ \frac{1}{n} : n \in \N , n \geq 1\}$ has Minkowski dimension one and Fact~\ref{fact:hasudorff} and Theorem~\ref{thm:FM} together imply that every subset of $\R$ definable in $(\rvec,D)$ either has interior or is nowhere dense.
It is shown in \cite{FHW-Compact} that if $E \subseteq \R^n$ is closed and the topological dimension of $E$ is strictly less than the Hausdorff dimension of $E$ then $(\rvec,E)$ defines every bounded Borel subset of every $\R^n$.

\medskip \noindent As a closed subset of $\R$ has topological dimension zero if it is nowhere dense and topological dimension one if it has interior, Theorem~\ref{thm:main} shows that if $E \subseteq \R$ is closed and the topological dimension of $E$ is strictly less than the coarse Minkowski dimension of $E$ then $(\rvec,E)$ defines every bounded Borel subset of every $\R^n$.
It is natural to conjecture that if $E \subseteq \R^n$ is closed and the topological dimension of $E$ is strictly less then the coarse Minkowski dimension of $E$ then $(\rvec,E)$ defines every bounded Borel subset of every $\R^n$.
In Theorem~\ref{thm:special} we will show as a corollary to Theorem~\ref{thm:main} that if $Z \subseteq \R^n$ is closed and has topological dimension zero and positive coarse Minkowksi dimension then $(\rvec,Z)$ defines every bounded Borel subset of $\R^n$.

\subsection*{Acknowledgements}
\noindent I thank the referee for many improvements and Philipp Hieronymi for useful discussions.

\section{Metric Notions}
\noindent
We recall two useful facts about $\mathscr{M}(\delta,X)$ and $\mathscr{N}(X)$, both of which are easy to see.
One can find more information about these invariants in Yomdin and Comte~\cite[Chapter 2]{YomdinComte} and many other places.

\begin{Fact}\label{fact:metric0}
Let $n \in \N$.
There are $K,L > 0$ such that for all bounded $X,Y \subseteq \R^n$ and $0 < \delta < \delta'$
$$ \mathscr{M}(\delta',X) \leq \mathscr{M}(\delta,X) \leq K \left( \frac{\delta'}{\delta} \right)^n \mathscr{M}(\delta',X) $$
and
$$ L^{-1} \mathscr{M}(\delta,X) \mathscr{M}(\delta,Y) \leq  \mathscr{M}(\delta, X \times Y) \leq L  \mathscr{M}(\delta, X) \mathscr{M}(\delta, Y)  $$
In particular
$$ L^{-1} \mathscr{M}(\delta,X)^2 \leq \mathscr{M}(\delta, X^2) \leq L \mathscr{M}(\delta,X)^2 $$
for all bounded $X \subseteq \R^n$.
\end{Fact}

\medskip \noindent The proof of the fact below is a straightforward computation that is essentially the same as the proof of the analogous fact for Minkowski dimension.
We leave the proof to the reader.

\begin{Fact}\label{fact:metric1}
For any $X \subseteq \R^n, Y \subseteq \R^m$ and $k \in \N$ we have
$$ \dimcm( X \times Y) \leq \dimcm(X) + \dimcm(Y) $$
and
$$ \dimcm(X^k) = k\dimcm(X).$$
\end{Fact}

\medskip \noindent Suppose that $X \subseteq \R^n$, $Y \subseteq \R^m$, $f$ is a map $X \to Y$, and $\lambda,\delta > 0$.
Then $f$ is a $(\lambda,\delta)$-quasi-isometry if
$$ \frac{1}{\lambda} \| x - x' \| - \delta \leq \| f(x) - f(x') \| \leq \lambda \| x - x' \| + \delta \quad \text{for all} \quad x,x' \in X, $$
and if for every $y \in Y$ we have $\| f(x) - y \| < \delta$ for some $x \in X$.
We say that $f$ is a quasi-isometry if it is a $(\lambda,\delta)$-quasi-isometry for some $\lambda,\delta > 0$.
It is well-known and easy to see that if there is a quasi-isometry $X \to Y$ then there is also a quasi-isometry $Y \to X$.
A map $g : X \to \R^n$ is a quasi-isometric embedding if it yields a quasi-isometry $X \to g(X)$.

\begin{Lem}\label{lem:cmqi}
Suppose $X \subseteq \R^n$, $Y \subseteq \R^m, 0 \in X, 0 \in Y$, and $f : X \to Y$ is a quasi-isometry such that $f(0) = 0$.
Then $X$ and $Y$ have the same coarse Minkowski dimension.
\end{Lem}

\noindent 
Lemma~\ref{lem:cmqi} holds without the assumptions that $0 \in X, 0 \in Y$, and $f(0) = 0$.
We do not prove this more general result to avoid technicalities.

\begin{proof}
We show that $\dimcm(Y) \leq \dimcm(X)$.
As there is a quasi-isometry $Y \to X$ that also maps $0$ to $0$ the same argument yields the other inequality.
Fix $\lambda,\delta > 0$ such that $f$ is a $(\lambda,\delta)$-quasi-isometry.

Fix $r > 0$.
Let $X(r) = B_n(0,r) \cap X$ and $Y(r) = B_m(0,r) \cap Y$.
Let $\{ B_n(p_i,1) \}_{i = 1}^{k}$ be a minimal covering of $X(r)$ by balls with radius $1$.
Then $\{ f ( B_n(p_i,1) ) \}_{i = 1}^{k}$ covers $f(X(r))$.
Let $q_i = f(p_i)$ for all $i$.
As $f$ is a $(\lambda,\delta)$-quasi-isometry we see that $f(B_n(p_i,1))$ is contained in $B_m(q_i, \lambda + \delta)$ for all $i$.
So $\{ B_m(q_i, \lambda + \delta)\}_{i = 1}^{k}$ covers $f(X(r))$.

We now show that every point in $Y(r
\lambda^{-1} - 2\delta )$ lies within distance $\delta$ of $f(X(r))$.
Fix $y \in Y(r\lambda^{-1} - 2\delta)$.
As $f$ is a $(\lambda,\delta)$-quasi-isometry there is $x \in X$ such that $\| f(x) - y \| < \delta$.
Suppose $\| x \| > r$.
Then as $f(0) = 0$ we have
$$ \|f(x)\| \geq \frac{1}{\lambda}\|x\| - \delta > r\lambda^{-1} - \delta.  $$
As $\| f(x) - y \| < \delta$ the triangle inequality yields $\|y\| > r \lambda^{-1} - 2\delta$.
Contradiction.

Combining the previous paragraphs we see that $\{ B_m( q_i, \lambda +2\delta) \}_{i = 1}^{k}$ covers $Y(r\lambda^{-1} - 2\delta)$.
Thus
$$ \mathscr{M}(\lambda + 2\delta, Y(r\lambda^{-1} - 2\delta)) \leq \mathscr{M}(1,X(r)) \quad \text{for all} \quad r > 0. $$
Applying the first claim of Fact~\ref{fact:metric0} we obtain a constant $L > 0$ depending only on $m$ such that
$$ L \mathscr{M}(1, Y(r\lambda^{-1} - 2\delta) ) \leq \mathscr{M}(\lambda + 2\delta, Y(r\lambda^{
-1} - 2\delta)) $$
hence
$$ L \mathscr{M}(1, Y(r\lambda^{-1} - 2\delta) ) \leq \mathscr{M}(1,X(r)). $$
Taking logarithms of of both sides of the expression above, dividing both sides by $\log(r)$, and taking the limit as $r \to \infty$ we see that $\dimcm(Y) \leq \dimcm(X)$.
\end{proof}

\section{Proof of Theorem~\ref{thm:wedge}}
\noindent 
Let $\mathbb{S}$ be the unit circle in $\R^2$.
Given $u \in \mathbb{S}$ we let $T_u : \R^2 \to \R$ be the orthogonal projection parallel to $u$, i.e., $T_u$ is the orthogonal projection such that $T_u( x ) = T_u( y )$ if and only if $x - y = tu$ for some $t \in \R$.
For our purposes a \textbf{double wedge} around $u \in \mathbb{S}$ is a subset of $\R^2$ of the form 
$$ C^{u}_{s,\varepsilon} := \{ tv : t \in \mathbb{R},|t| > s, v \in \mathbb{S}, \|v - u\| < \varepsilon \} $$
for some $s,\varepsilon > 0$.

\begin{Lem}\label{lem:wedge}
Let $F$ be a nonempty subset of $\R^2$ and $u \in \mathbb{S}$.
If $F - F = \{ x - y : x,y \in F \}$ is disjoint from some double wedge around $u$ then the restriction of $T_u$ to $F$ is a quasi-isometric embedding $F \to \R$.
\end{Lem}

\noindent
Lemma~\ref{lem:wedge} is a quasi-isometric version of a well-known fact from geometric measure theory: if $F$ is a nonempty subset of $\R^2$ such that $F - F$ is disjoint from a double wedge of the form $C^u_{\varepsilon,0}$ then the restriction of $T_u$ to $F$ is a bilipschitz embedding $F \to \R$.
This fact is applied in \cite{FHM, HM}.

\begin{proof}
Suppose that $F - F$ is disjoint from $C^{u}_{s,\varepsilon}$.
As $T_u$ is an othogonal projection we have $\| T_u(x) - T_u(x') \| \leq \| x - x' \|$ for all $x,x' \in \R^2$, so it suffices to obtain a lower bound on $\| T_u(X) - T_u(x')\|$ of the appropriate form.

After making a change of coordinates if necessary we suppose $u = (0,1)$ so that $T_u(x,y) = x$ for all $(x,y) \in \R^2$.
Then we have 
$$ C^{u}_{s,\varepsilon} = \{ (x,y) \in \R^2 : |y| > \lambda |x| \quad \text{and} \quad \|(x,y)\| > s \} $$
for some $\lambda > 0$ depending only on $\varepsilon$.
Thus, if $(x,y) \in F - F$ then either $\|(x,y)\| < s$ or $|y| \leq \lambda |x|$.
Equivalently, for all $(x,y), (x',y') \in F$ we either have 
$$ \| (x,y) - (x',y') \| < s \quad 
\text{or} \quad | y - y'| \leq \lambda | x - x' |. $$
In the latter case we have
$$ \| (x,y) - (x',y') \| \leq |x - x'| + |y - y'| \leq (1 + \lambda)| x - x'|  $$
hence
$$ \frac{1}{1 + \lambda}\|(x,y) - (x',y') \| \leq |x - x'|. $$
In the first case we have
$$ \| (x,y) - (x',y') \| - s < |x - x'|. $$
So for all $(x,y),(x',x') \in F$ we have
$$ \frac{1}{1 + \lambda} \| (x,x') - (y,y') \| - s \leq |x - x'|. $$
So the restriction of $T_u$ to $F$ is a quasi-isometric embedding $F \to \R^2$.
\end{proof}

\noindent We let $\mathbb{H}$ be the upper half plane $\{ (x,y) \in \R^2 : y > 0 \}$ and let $\mathbb{S}^+ = \mathbb{S} \cap \mathbb{H}$.
A wedge in $\mathbb{H}$ around $u \in \mathbb{S}^+$ is a set of the form
$$ C^{u,+}_{s,\varepsilon} :=  \{ tv : t \in \mathbb{R}, t > s, v \in \mathbb{S}, \|v - u\| < \varepsilon \} $$
such that $C^{u,+}_{s,\varepsilon} \subseteq \mathbb{H}$.

\begin{Lem}\label{lem:wedge1}
Suppose $F \subseteq \mathbb{H}$ intersects every wedge in $\mathbb{H}$.
Then there is a $u \in \mathbb{S}^+$ such that $T_u(F)$ is dense.
\end{Lem}

\noindent The reader may find that drawing a few pictures greatly assists in comprehending the proof of Lemma~\ref{lem:wedge1}.
We let $p = (-1,0)$ and $o = (0,0)$.
Note that if $z \in \mathbb{H}$, $q$ is a positive real number, and $u \in \mathbb{S}^+$, then $T_u(z) = q$ if and only if $\angle pou = \angle pqz$.

\begin{proof}
We show that the set of $u \in \mathbb{S}^+$ such that $T_u(F)$ is dense in $\mathbb{R}$ is comeager in $\mathbb{S}^+$.
It suffices to show that
$$ \{ u \in \mathbb{S}^+ : T_u(F) \cap I \neq \emptyset \} $$
is open and dense in $\mathbb{S}^+$ for every nonempty open interval $I$ with rational endpoints.
Fix a nonempty open interval $I = (q_1,q_2)$ with rational endpoints.
We suppose that $q_1,q_2 > 0$ for the sake of simplicity, the more general case follows by trivial modifications of our argument.
The map $T : \mathbb{S}^+ \times \mathbb{R}^2 \to \mathbb{R}$ given by $T(u,x) = T_u(x)$ is continuous.
Thus if $T_u(x) \in I$ then $T_v(x) \in I$ for all $v \in \mathbb{S}^+$ sufficiently close to $u$.
It follows that the set of $u$ such that $T_u(F) \cap I \neq \emptyset$ is open in $\mathbb{S}^+$.

It now suffices to show that the set of $w \in \mathbb{S}^+$ such that $T_w(F) \cap I \neq \emptyset$ is dense in $\mathbb{S}^+$.
Fix $u,v \in \mathbb{S}^+$ such that $\angle pou < \angle pov$ and let $J$ be the set of $w \in \mathbb{S}^+$ such that $ \angle pou < \angle pow < \angle pov $.
We show there is a $w \in J$ such that $T_w(F) \cap I \neq \emptyset$.
Let $r_1,r_2 \in \mathbb{H}$ be such that $\angle pq_1r_1 = \angle pou$ and $\angle pq_2r_2 = \angle pov$.
Let $D$ be the set of points in $\mathbb{H}$ that lie in between the rays $\overrightarrow{q_1r_1}$ and $\overrightarrow{q_2r_2}$.
It is easy to see that
$$ D = \bigcup_{q \in I} \{ r \in \mathbb{H} : \angle pou < \angle pqr < \angle pov \} = \bigcup_{ q \in I } \bigcup_{ w \in J } T_w^{-1}( \{ q \} ) = \bigcup_{w \in J} T_w^{-1}(I). $$
It therefore suffices to show that $D$ intersects $F$.
Let $z_1,z_2 \in \mathbb{S}^+$ be such that 
$$ \angle pou < \angle poz_1 < \angle poz_2 < \angle pov.  $$
As $\angle pq_1r_1 < \angle poz_1 < \angle poz_2 < \angle pq_2r_2$, we see that every element of $\overrightarrow{oz_1}$ or $\overrightarrow{oz_2}$ sufficiently far from the origin lies in $D$.
It follows that there is a $t > 0$ such that
$$ W := \{ z \in \mathbb{H} : \|z\| \geq t, \angle poz_1 < \angle poz < \angle poz_2 \} \subseteq D. $$
Then $W$ is a wedge in $\mathbb{H}$ and so contains an element of $F$.
Thus $D$ contains an element of $F$.
\end{proof}

\begin{Lem}\label{lem:wedge3}
Suppose $E \subseteq \R$.
Then one of the following holds:
\begin{enumerate}
\item there is  a $u \in \mathbb{S}$ such that the restriction of $T_u$ to $E^2$ is a quasi-isometric embedding $E^2 \to \R$,
\item there is a linear $S : \R^4 \to \R$ such that $S(E^4)$ is dense.
\end{enumerate}
\end{Lem}

\begin{proof}
Consider $E^2 - E^2 \subseteq \R^2$.
If $E^2 - E^2$ is disjoint from a double wedge in $\R^2$ then Lemma~\ref{lem:wedge1} shows that some $T_u$ quasi-isometrically embeds $E^2$ into $\R$.

Suppose $E^2 - E^2$ intersects every double wedge in $\R^2$.
Note that if $(x,y) \in E^2 - E^2$ then $(-x,-y)$ is also an element of $E^2 - 
E^2$.
It is easy to see that this implies that $E^2 - E^2$ intersects every wedge in $\mathbb{H}$.
Applying Lemma~\ref{lem:wedge3} we fix a $u \in \mathbb{S}$ such that $T_u(E^2 - E^2)$ is dense.
Let $S : \R^4 \to \R$ be the linear function given by
$$ S(x,y,x',y') = T_u(x - x', y-y') \quad \text{for all} \quad x,y,x',y' \in \R. $$
Then $S(E^4)$ is dense.
\end{proof}


\noindent We now prove Theorem~\ref{thm:wedge}.

\begin{proof}
Suppose towards a contradiction that $E \subseteq \R$ has positive coarse Minkowski dimension and $T(E^n)$ is not dense for every $n \in \N$ and linear $T : \R^n \to \R$.
We may suppose that $0 \in E$.
Let $\mathcal{S}$ be the collection of sets of the form $T(E^n)$ for linear $T : \R^n \to \R$.
It is easy to see that if $F \in \mathcal{S}$ and $T: \R^n \to \R$ is linear then $T(F^n)$ is also in $\mathcal{S}$.
We let $s$ be the supremum of the coarse Minkowski dimensions of members of $\mathcal{S}$.
Every element of $\mathcal{S}$ has coarse Minkowski dimension $\leq 1$, so $s$ exists and $s \leq 1$.
As $\dimcm(E) > 0$ we have $s > 0$.
Let $F \in \mathcal{S}$ be such that $\dimcm(F) > \frac{1}{2}s$.
An application of Lemma~\ref{lem:wedge3} yields a linear $T : \R^2 \to \R$ such that the restriction of $T$ to $F^2$ is a quasi-isometric embedding $F^2 \to \R$.
Lemma~\ref{lem:cmqi} and Fact~\ref{fact:metric1} together show that
$$ \dimcm T(F^2) = \dimcm(F^2) = 2\dimcm(F) > s. $$
But $T(F^2) \in \mathcal{S}$, contradiction.
\end{proof}


\section{A corollary in $\R^n$}
\noindent We prove a higher dimensional version of the second claim of Theorem~\ref{thm:main}.
(Recall that a closed subset of $\R^n$ has topological dimension zero if and only if it is nowhere dense.)

\begin{Thm}
\label{thm:special}
Suppose $Z$ is a closed subset of $\R^n$ with topological dimension zero.
If $Z$ has positive coarse Minkowski dimension then $(\rvec,Z)$ defines all bounded Borel subsets of all $\R^n$.
\end{Thm}

\begin{proof}
We suppose that $(\rvec,Z)$ does not define all bounded Borel subsets of all $\R^n$ and show that $\dimcm(Z) = 0$.
Given $1 \leq k \leq n$ we let $\pi_k : \R^n \to \R$ be given by
$$ \pi_k(x_1,\ldots,x_n) = x_k \quad \text{for all} \quad (x_1,\ldots,x_n) \in \R^n. $$
An application of \cite[Theorem D, Theorem E]{FHW-Compact} shows that $\pi_k(Z)$ is nowhere dense for all $1 \leq k \leq n$.
Theorem~\ref{thm:main} shows that $\dimcm \pi_k(Z) = 0$ for all $1 \leq k \leq n$.
As $Z$ is a subset of $\pi_1(Z) \times \ldots \times \pi_n(Z)$
repeated application of Fact~\ref{fact:metric1} shows that $\dimcm(Z) = 0$.
\end{proof}

\bibliographystyle{amsplain}
\bibliography{Ref}

\end{document}